\documentclass[11pt]{article}
\usepackage{a4wide}
\usepackage{color}
\usepackage{graphicx}
\usepackage{a4wide}
\usepackage{amsmath,amsthm,hyperref}
\usepackage{amsfonts}
\usepackage{amssymb}
\usepackage{url}
\usepackage{setspace}

\newtheorem{theorem}{Theorem}[section]
\newtheorem{lemma}{Lemma}[section]

\newtheorem{remark}{Remark}[section]
\newtheorem{definition}{Definition}[section]

\def \oet{\Omega^{\varepsilon T}}			
\def \get{\Gamma_G^{\varepsilon T}}               
\def \gee{\Gamma_G^{\varepsilon}}               
\def \ep{\varepsilon}                           

\numberwithin{equation}{section}
\title{Homogenization of a pore scale  model for precipitation and dissolution in porous media}
\author{K. Kumar$^1$ , M. Neuss-Radu$^2$ , I. S. Pop$^{3}$  \\
{$^1$ Center for Subsurface Modeling, ICES, UT Austin, Texas, U.S.A.} \\
{$^2$ Department Mathematik, Universit\"at Erlangen-N\"urnberg, Erlangen, Germany}\\
{$^3$ Dept. of Math. and Comp. Sci., Eindhoven University of Technology, The Netherlands}\\
{\emph{kkumar@ices.utexas.edu, maria.neuss-radu@math.fau.de, i.pop@tue.nl}}
}
\date{}

\begin{document}
\maketitle
\begin{abstract}

In this paper we employ homogenization techniques to provide a rigorous derivation of the Darcy scale model for precipitation and dissolution in porous media proposed in \cite{knabner}. The starting point is the pore scale model in \cite{porescale}, which is a coupled system of evolution equations, involving a parabolic equation and an ordinary differential equation. The former models ion transport and is defined in a periodically perforated medium. It is further coupled through the boundary conditions to the latter, defined on the boundaries of the perforations and modelling the dissolution and precipitation of the precipitate.

The main challenge is in dealing with the dissolution and precipitation rates, which involve a monotone but multi-valued mapping. Due to this, the micro-scale solution lacks regularity. With $\ep$ being the scale parameter (the ratio between the micro scale and the macro scale length), we adopt the 2-scale framework to achieve the convergence of the homogenization procedure as $\ep$ approaches zero.
\end{abstract}

\section{Introduction}
In this paper, we employ rigorous homogenization \index{Homogenization!Periodic} techniques to derive the effective (Darcy scale) model for dissolution and precipitation in a complex (porous) medium proposed in \cite{knabner}. The starting point is the micro (pore) scale model analyzed in \cite{porescale, tycho3}, where the existence and uniqueness of a solution are proved. The particularity is in the dissolution and precipitation, involving multivalued rates. Using homogenization techniques, here we give a rigorous derivation of the macro (core) scale counterpart. For the resulting upscaled model existence and uniqueness is obtained.

At the micro scale, the medium consists of periodically repeating solid grains surrounded by voids (the pores). The pore space forms a periodically perforated domain (the grains being the perforations in the domain) which is completely filled by a fluid (e.g. water). The fluid is flowing around the solid grains, transporting solutes, which are dissolved ions. The solute may further diffuse in the fluid, whereas at the the grain surfaces (the boundaries of the perforations), the solute species may react and precipitate, forming a thin layer of an immobile species (salt) attached to these boundaries. The reverse process of dissolution is also possible.

One important assumption is that the layer of the species attached to the grain boundaries (the precipitate) is very thin when compared to the pore thickness, so eventual changes in the the geometry at the pore-scale can be neglected. This allows decoupling the equations modelling the flow from those describing the chemical processes. This assumption is justified whenever the density of the deposited layer is very large when compared to the typical density  of the solute (see \cite{kumar, tycho2, tycho, tycho-sorin-stefan-problem}). These papers consider the alternative approach, where the precipitate layer induces non-negligible changes in the pores, leading to a model involving free boundaries at the micro scale.

Encountered at the boundary of the perforations, the precipitation process is modeled by a rate function that is monotone and Lipschitz continuous with respect to the solute concentrations. This is consistent with the mass action kinetics. For the dissolution, at sites on the grain boundary where precipitate is present, it will be dissolved at a constant rate. A special situation is encountered when no precipitate is present at one site, when certainly no dissolution is possible. {Besides, at such a location} a precipitate layer (meaning an effective occurrence of the immobile species) is only possible if the fluid is "oversaturated". This means that the precipitation rate exceeds a threshold value, the so-called solubility product.
In the "undersaturated" regime, when the precipitation rate is below the solubility product, no effective gain in the precipitate is possible. This can be seen as an instantaneous dissolution of any precipitate formed in undersaturated conditions, so the overall result of these processes encountered at the time scale of interest is null. {In other words, the precipitation rate is in balance with the dissolution rate.} Between oversaturation and undersaturation, the precipitation rate equals the solubility product, which is an {equilibrium value}. In this case neither precipitation, nor dissolution is encountered.

Note that the undersaturated regime is encountered for any value of the precipitation rate that is below the solubility product. Since the overall rate is zero, at sites where no precipitate is present, the dissolution rate should take a value between zero (no dissolution) and the equilibrium one (the solubility product), in order to balance the dissolution rate. To model this situation, we define the dissolution rate as a member of a multi-valued graph \index{Heaviside graph}(a scaled Heaviside graph).
The macro scale model for the present problem has been proposed in \cite{knabner} and further discussed in \cite{hans_peter1,hans_peter2,hans_peter3}, where the main focus is related to travelling waves. Its pore scale counterpart has been analyzed in \cite{porescale} and \cite{tycho3}, where existence and uniqueness results are obtained. Furthermore, in \cite{porescale} a two dimensional strip was considered as a model geometry for deriving rigorously the macro scale model by a transversal averaging procedure.

Still for a simplified geometry, but for the case when free boundaries are encountered at the pore scale due to dissolution and precipitation, upscaled models are derived formally in \cite{tycho} for moderate {\it Pecl\'et} numbers. The same situation, but now under a dominated transport regime - high {\it Pecl\'et} numbers, is considered in \cite{kumar}. The upscaled model is similar to Taylor dispersion, but now includes the effect of the changing geometry and of the reactions at the micro scale. Similar models are also obtained in \cite{tycho-sorin-rainer} for biofilm growth, in \cite{Ray} for drug release from collagen matrices and in \cite{Peter} for a reactive flow model involving an evolving microstructure.

The dissolution and precipitation model under discussion here was considered also in domains with rough boundaries. Assuming that the precipitate does not affect the domain, effective boundary conditions are derived rigorously in \cite{KvHPop}. Similar results, but for the alternative approach involving free and rough boundaries are obtained formally in \cite{KvNPop}.

Strictly related to the dissolution and precipitation model discussed here, we recall that the convergence of numerical schemes is analyzed in \cite{devigne} for the micro scale model, and in \cite{kumar_fem,KPR_MFEM} for the macro scale model. We further refer to \cite{AdrianMaria}, analyzing a multiscale Galerkin approach to couple the micro scale and the macro scale variables. Though the results are for Lipschitz-type nonlinearities, this method can be adapted in this context too.

The above mentioned rigorous upscaling result was obtained in \cite{porescale} in a simplifed setting: a two-dimensional strip. In this case, a simple transversal averaging procedure can be applied. Here {we consider the more general situation, when the porous medium is modelled by a periodically perforated domain. Clearly, this requires a different upscaling approach. For the rigorous derivation of the macroscopic model,} we use the 2-scale convergence concept developed in \cite{Allaire,Ng89} and extended further in \cite{Maria} to include model components defined on lower dimensional manifolds (the grain boundaries). In the limit, the resulting upscaled model has the same structure as the model proposed in \cite{knabner}.

We mention \cite{ Jaeger,jaeger-mikelic} for pioneering works on rigorous homogenization of reactive flow models, including  the derivation of upscaled models from well-posed microscopic (pore-scale) models. Since then many publications have considered similar problems; we restrict here to mention papers that are very close to the present contribution. Non-Lipschitz but continuous reaction rates are considered \cite{Conca}, but for one species. In \cite{Maria} and \cite{anna} the two-scale convergence framework is extended for variables defined on lower-dimensional manifolds. Rigorous homogenization results for reactive flows including adsorption and desorption at the boundaries of the perforations, but in dominating flow regime (high Pecl{\' e}t numbers) are obtained in \cite{AMP10,AP10, AndrovanDuijn}. The two-scale convergence approach has been extended to include the mechanics of the porous media and finds application in several fields including the biological, mechanical etc. A recent work dealing with combining the reactive flow with the mechanics of cells is \cite{JMR11}. Of particular relevance to the present work is the work of \cite{anna} where non-linear reaction terms on the surface are treated using the techniques of periodic unfolding.

The major challenge in the present work is in dealing with the dissolution rates, which is a member of a multi-valued graph. For a proper interpretation of this rate, we first consider it as the limit of its regularized version. Following \cite{porescale}, this allows identifying the dissolution rate in a unique way. However, the resulting dissolution rate is non-Lipschitz and may even become discontinuous. This brings two difficulties in obtaining the rigorous results: compared to models involving Lipschitz continuous rates, the solution component defined in the perforated domain lacks regularity, and for the solution components defined on the boundary of the perforations
a proper convergence concept is required.

Specifically, for passing to the limit in the sequence of micro scale solutions one usually extends the solution components defined in the porous medium to the entire domain, including perforations. The convergence is then obtained by uniform energy estimates, which involve all (weak) derivatives. The estimates for the spatial derivatives are obtained here in the usual manner. The time derivative instead needs more attention. The common approach, similar to deriving the convection-diffusion-reaction equation with respect to time, does not work here due to the particular dissolution rate. Here we elaborate the ideas in \cite{KundanTh}, and show that the extension satisfies uniform estimates strictly in the space where the solution is defined, and not a better one. In the present context, this seems to be the optimal result.

For the components defined on the boundaries, we follow the ideas in \cite{maria-willi} and \cite{anna}, where the concept of strong two-scale convergence is introduced. This is based on unfolding/localization operators \cite{BLM96,CDG08}. In particular, for the immobile species (the precipitate) we  obtain compactness results of the unfolded sequence leading to the strong convergence. These results allow us to identify the limit of the (pore scale) dissolution rate.

\begin{figure}
\centering
\includegraphics[scale=0.5]{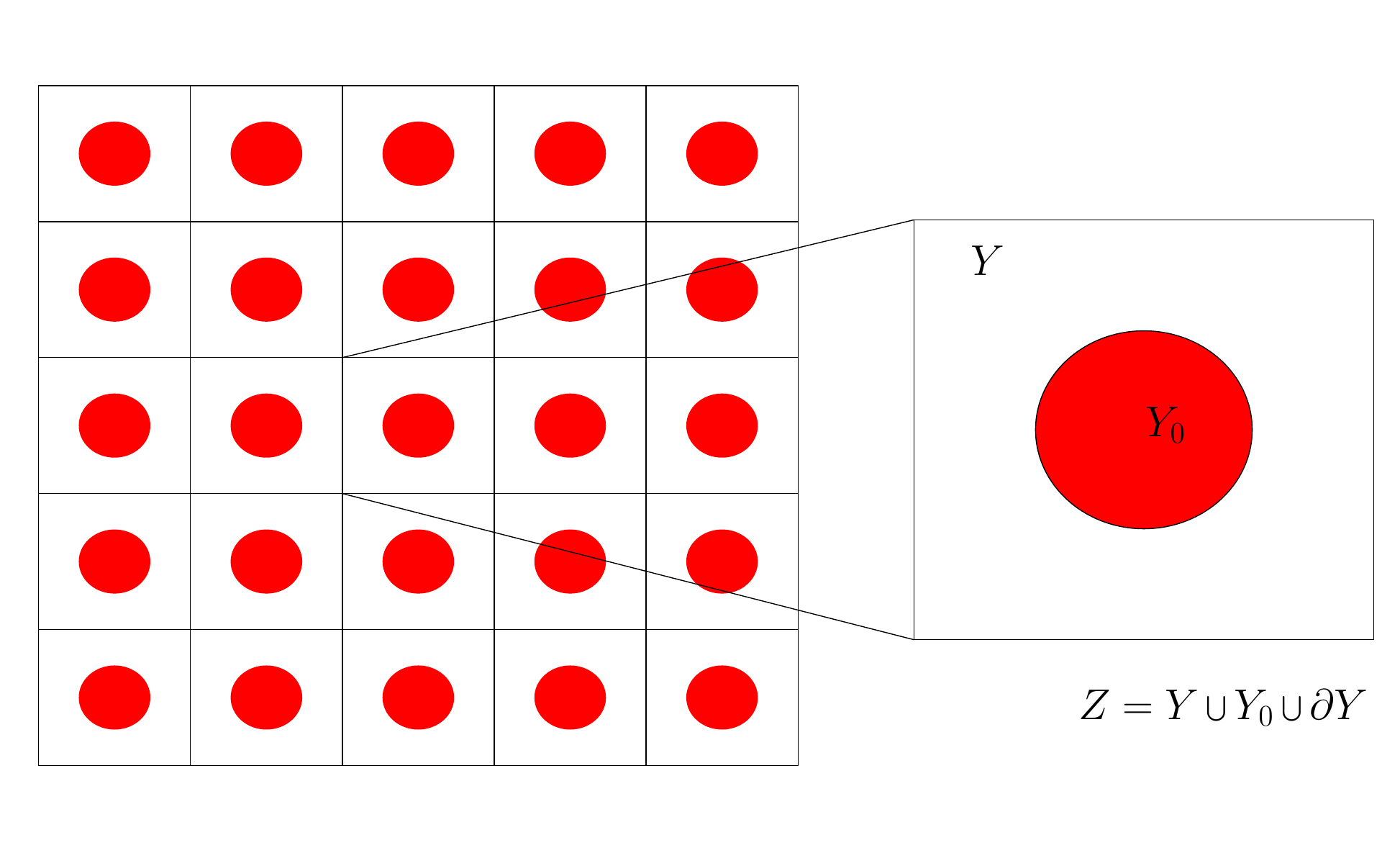}
\caption{Left: the porous medium $\Omega$ consisting of $\ep$-scaled perforated cells distributed periodically; the total void space is $\Omega^\ep$. Right: a reference cell containing the flow/transport part (the pore $Y$) and the perforation (the solid grain) $Y_0$ separated by the interface $\Gamma_G$. Note that the geometry remains fixed in time for a given $\ep$.} \label{fig:PH_schematic}
\end{figure}
\section{Setting of the model}
Let $\varepsilon > 0$ be a sequence of strictly positive numbers tending to zero, with the property that $\frac{1}{\varepsilon} \in \mathbb{N}$. Let $[0,T]$ denote a time interval, with $T>0.$

We consider the domain $\Omega = (0,1)^3 $ consisting of two subdomains: the perforations (representing the solid grains) and the the perforated domain (the pore space) filled with fluid and where flow, diffusion and transport is taking place, see Fig \ref{fig:PH_schematic}.
At the micro scale, the domain of interest (the fluid part) is denoted by $\Omega^\varepsilon$, and the boundary of the perforations by $\Gamma_G^\varepsilon$. The boundary of the domain $\Omega$ consists of two parts
$$
\partial \Omega =  \Gamma_D \cup \Gamma_N \quad \mbox{and} \quad \Gamma_D \cap \Gamma_N = \emptyset.
$$
The outer unit normal to $\partial \Omega$ is denoted by $\boldsymbol{\nu}.$ On $\Gamma_G^\varepsilon$, the boundaries of the perforations, $\boldsymbol{\nu}$ is the unit normal pointing into the perforations.

The microscopic structure of $\Omega^\varepsilon$ and $\Gamma_G^\varepsilon$ is periodic, and is obtained by the repetition of the standard cell $Z= (0,1)^3$ scaled with the small parameter $\ep$. We denote by $Y$ and $Y_0$ the fluid part, respectively the perforation in $Z$. On $\partial Y_0$, we denote by $\boldsymbol{\nu}$ the unit normal to $\Gamma_G$ pointing into the perforation $Y_0$. We assume that
\begin{enumerate}
\item $  \overline{Y}_0 \subset Z, \quad Y =Z \setminus  \overline{Y}_0$,
\item $Y_0$ is a set of strictly positive measure, with piecewise smooth boundary $\Gamma_G=\partial Y_0$.
\end{enumerate}
Let
$$
E_s:= \bigcup_{k \in \mathbb{Z}^3}  \overline{Y_0^k} = \bigcup_{k \in \mathbb{Z}^3} \overline{(Y_0 +k) }.
$$
Then the fluid part of the porous medium $\Omega^\varepsilon$ and the total boundary of the perforations $\Gamma_G^\varepsilon$ are defined as follows:
$$
\Omega^\varepsilon = \Omega \setminus \varepsilon E_s, \quad \Gamma_G^\varepsilon = \Omega \cap \varepsilon \bigcup_{k \in \mathbb{Z}^3} \partial Y_0^k.
$$
We emphasize that the assumption of $\Omega$ being a unit cube can be slightly generalized. The results hold also for domains $\Omega$ with the property that for each $\ep$, there exists  $I^\ep \subset \mathbb{R}^3$ such that
\begin{align*}
\bar{\Omega} = \bigcup \{\ep Z^k: k \in I^\ep\}.
\end{align*}
This means that the domain $\Omega$ the union of entire cells, for all chosen values of $\ep$.

Finally, for any $t \in (0, T]$ we define
$$Q^t = (0, t] \times Q,$$
where $Q$ is one of the sets $\Omega$, $\Omega^\ep$, $\Gamma_G$, or $\Gamma_G^\ep$.

\subsection{The micro scale model} \label{sec:micromodel}
Let us now formulate the equations which model the processes at the microscopic level.
The microscopic model contains two components: the equations  for the flow, and the equations for the chemistry. For the flow, we consider the Stokes system  \index{Stokes equation}
\begin{align}\label{eq:PH:Stokes}
\left\{
\begin{array}{rcl}
\varepsilon^2 \triangle  \textbf{q}^\varepsilon &=&\nabla P^\varepsilon, \\
\nabla \cdot \textbf{q}^\varepsilon &=& 0,
\end{array}
\right.
\end{align}
for all $x \in \Omega^\ep$.
In the above, $\textbf{q}^\ep$ stands for the fluid velocity, $P^\ep$ denotes the pressure inside the fluid. With a proper scaling, when bringing the model to a dimensionless form the dynamic viscosity becomes $\ep^2$ (see e.g. \cite{Hornung}, p. 45). We complement Stokes equations by assigning no-slip boundary conditions at the boundary of the perforations and given Dirichlet boundary conditions at the outer boundary $\partial \Omega$,
\begin{align}\label{eq:PH:StokesBC}
\textbf{q}^\varepsilon = 0, \; \mbox{ on  } \Gamma_G^\varepsilon, \qquad \mbox{ and } \qquad
\textbf{q}^\varepsilon = \textbf{q}_D, \; \mbox{ on  } \partial \Omega,
\end{align}
where $\textbf{q}_D$ is such that $\displaystyle \int\limits_{\partial \Omega}\mathbf{\boldsymbol{\nu}} \cdot \textbf{q}_D = 0$.
As mentioned above, we assume that the chemical processes neither change the micro scale geometry, nor the fluid properties. Therefore the flow component does not depend on the other components of the model, and can be completely decoupled. This means that one can solve first the Stokes system \eqref{eq:PH:Stokes} with the given boundary conditions \eqref{eq:PH:StokesBC} to obtain the fluid velocity $\textbf{q}^\varepsilon$. We further assume that $\textbf{q}^\varepsilon$ is essentially bounded uniformly w.r.t. $\varepsilon$, i.e.
\begin{equation}\label{eq:PH:M_q}
\|\textbf{q}^\varepsilon\|_{\infty, \Omega} \leq M_q < \infty
\end{equation}
for some constant $M_q > 0$. For the Stokes model with homogeneous Dirichlet boundary conditions,
the essential boundedness of $\textbf{q}^\varepsilon$ holds if, for example, the domain is polygonal
(see \cite{KO,KK}). Here we assume that this estimate is uniform in $\ep$. Since the focus here is on the chemistry and recalling that the flow component is an independent one, in what follows we simply assume $\textbf{q}^\ep$ given, having the properties mentioned above.  

The main interest in this paper is in the subsystem modeling the chemical processes. This takes into account two solute (mobile) species, which are transported by the fluid. In the fluid, these species are diffusing, but no reactions are taking place there. The corresponding model is therefore a convection-diffusion equation in the fluid part $\Omega_\varepsilon$. Following \cite{knabner, porescale}, we simplify the analysis by considering only one immobile species, having the concentration $u^\ep$. This is justified if the two species are having the same diffusion coefficient. Accounting for more species is fairly straightforward.

The chemical processes are encountered at the boundary of perforations, where the mobile species react forming the precipitate. The reaction result is an immobile species (the precipitate) attached to this boundary and having the concentration $v^\ep$. The precipitate may be dissolved, becoming a source of mobile species in the fluid. In the mathematical model, the precipitation and dissolution are rates in the ordinary differential equation defined in every location on the boundary of perforations. Finally, the partial differential equation and the ordinary one are coupled through the boundary conditions.

With the concentrations $u^\ep$ and $v^\ep$ introduced above, the chemistry is described by the following equations
\begin{align}
\label{eq:PH:om} \left \{
\begin{array}{rcll}
\partial_t u^\varepsilon + \nabla \cdot  (\textbf{q}^\ep u^\varepsilon - D \nabla u^\varepsilon) &=& 0, \quad & \text{in} \quad \Omega^{\varepsilon T},\\
-D \boldsymbol{\nu} \cdot \nabla u^\varepsilon &=& \varepsilon n \partial_t v^\varepsilon, \quad & \text{on} \quad \Gamma_G^{\varepsilon T},\\
\partial_t v^\varepsilon &=& k (r(u^\varepsilon)-w^\varepsilon), \quad & \text{on} \quad \Gamma_G^{\varepsilon T},\\
w^\varepsilon  &\in& H(v^\varepsilon), \quad & \text{on} \quad \Gamma_G^{\varepsilon T}.
\end{array}
\right.
\end{align}
The system \eqref{eq:PH:om} is complemented by the following initial and boundary conditions, 
\begin{align}
\label{eq:PH:ic} \left \{ \begin{array}{lll}
u^\varepsilon(0,\cdot) &=& u_I \quad \text {in} \quad \Omega^{\varepsilon },\\
v^\varepsilon(0,\cdot) &=& v_I \quad \text{on} \quad \Gamma_G^{\varepsilon },\\
u^\varepsilon &=& 0, \quad \text{on} \quad \Gamma_D^T, \\
(\textbf{q}^\ep u^\varepsilon - D \nabla u^\varepsilon) \cdot \boldsymbol{\nu} &=& 0, \quad \text{on} \quad \Gamma_N^T.
\end{array}
\right.
\end{align}
As mentioned above, $\textbf{q}^\ep$ solves the Stokes system \eqref{eq:PH:Stokes}-\eqref{eq:PH:StokesBC}, which is not affected by the chemistry and therefore we assume it given. Hence, the unknowns of the microscopic model are $u^\ep$, $v^\ep$, and $w^\ep$.  In particular, $w^\ep$ describes the dissolution rate; the specific choice in \eqref{eq:PH:om}$_4$ will be explained below. Note that $u^\ep$ is defined in the domain $\Omega^\ep$, while $v^\ep$ and $w^\ep$ are defined on the boundaries of perforations, $\Gamma^\ep_G$. The physical constant $D$ is a (given) diffusion coefficient, assumed constant. Further, $k$ is a dimensionless reaction rate constant, which we assume of moderate order w.r.t. $\ep$ and is normalized to $1$. In physical sense, this means that the precipitation sites are homogeneous. The extension to the non-homogeneous case requires some additional technical steps in the proofs but does not pose any major difficulties. Also note that assuming that $D$ and $k$ are moderate w.r.t. $\ep$ implies that the time scales of diffusion, transport and chemical processes are of the same order. Finally, $n$ is a constant denoting the valence of the solute and for simplicity, we will be taking it as $1$.

Clearly, \eqref{eq:PH:om}$_2$ relates the change in the precipitate to the normal flux of the ions at the boundaries, assuming the no-slip boundary condition for $\textbf{q}^\ep$. Also observe the appearance of $\ep$ in the boundary flux. As will be seen below, this allows to control the growth of the precipitate when passing to the limit in the homogenization step. We refer to Chapter 1 of \cite{Hornung} for a justification of this choice based on the geometry of the pores, and to \cite{porescale}, Remark 1.2 for an equivalent interpretation.

We proceed now by explaining the precipitation rate $r(u^\ep)$ and the dissolution rate $w^\ep$ appearing in the last two equations of \eqref{eq:PH:om}.
We assume first that
\begin{enumerate}
\item[] The precipitation rate $r$ depends on the solute concentration, where
\begin{align} \label{eq:PH:assump1} \tag{A.1}
r:\mathbb{R} \rightarrow [0,\infty) \text{ is locally Lipschitz in} \quad \mathbb{R}.
\end{align}
\item[] There exists a unique $u_* \geq 0$, such that
\begin{align} \label{eq:PH:assump2} \tag{A.2}
r(u^\varepsilon) = \left \{ \begin{array}{lll}
0 \quad \text{for} \quad u^\varepsilon \leq u_*,\\
\text{strictly increasing for} \quad u^\varepsilon \geq u_* \quad \text{with}  \quad r(\infty) = \infty.
\end{array}
\right.
\end{align}
\end{enumerate}
An example where these assumptions are fulfilled is given in \cite{knabner}, where a model based on mass-action kinetics is considered. Note that a value $u^* > 0$ exists such that
$$
r(u^*) = 1 .
$$
With the proper scaling, this value is exactly the solubility product mentioned in the introduction. As explained, this value is taken at an equilibrium concentration: if $u^\ep(t, x) = u^*$, neither precipitation, nor dissolution is encountered in $x$ at time $t$.

Finally, the dissolution rate satisfies
$$
w^\ep \in H(v^\varepsilon) ,
$$
where $H(\cdot)$ denotes the Heaviside graph,
$$
H(u) = \left\{
\begin{array}{ll}
\{0\} & \mbox{ if } u < 0, \\
{ [0 , 1 ] } & \mbox{ if } u = 0, \\
\{1\} & \mbox{ if } u > 0.
\end{array}
\right .
$$
This means that whenever precipitate is present, hence $v^\ep(t, x) > 0$, in this point dissolution is encountered at a constant rate, 1 by scaling. One may view this as a surface process: it does not matter how much precipitate is present in one location $x$ on the boundary of perforations at  some time $t$, the dissolution will be encountered strictly at the surface of the precipitate and not in the interior. A more interesting situation appears at sites where the precipitate is absent, thus $v^\ep(t, x) = 0$. Then a value has to be specified for the dissolution rate $w^\ep(t, x) \in [0, 1]$. To important features should be accounted for: no dissolution is allowed whenever precipitate is absent, and further no precipitation should be encountered in the {\it undersaturated} regime, when $u^\ep(t, x) < u^*$. As explained in \cite{porescale, knabner, tycho3}, whenever $v^\ep = 0$, the rate $w^\ep$ depends also on the solute concentration $u^\ep$ at the boundary. Specifically, in the {\it oversaturated} regime, when $u^\ep > u^*$ (the value $u^*$ being introduced above) we take $w^\ep = 1$. Since $r(u^\ep) > 1$, this means that the overall precipitation/dissolution rate is strictly positive, resulting in a net gain in the precipitate. In the undersaturated regime one hase $u^\ep < u^*$, thus $r(u^\ep) \leq 1$. Then the solute concentration cannot support an effective gain in precipitate, and the overal rate remains 0. In particular, dissolution should be avoided in this case. To achieve this, we take $w^\ep = r(u^\ep) \in [0, 1)$ and the overall rate becomes 0. Finally, since $r(u^*) =1$, the case $u^\ep = u^*$ leads to an equilibrium, regardless of the value of $v^\ep$. This can be summarized as
\begin{align}\label{eq:PH:w=r}
w^\ep = \left\{
\begin{array}{ll}
0 \qquad &\mbox{ if } v^\ep < 0, \\
\min\{r(u^\ep), 1\} &\mbox{ if } v^\ep = 0, \\
1 \qquad &\mbox{ if } v^\ep > 0.
\end{array}
\right .
\end{align}
The dissolution rate is defined for unphysical, negative values of $v^\ep$ for the sake of completeness. We will prove below that whenever the initial precipitate concentration is non-negative, no negative concentrations can be obtained. Note that in the above relation, $w^\ep \in H(v^\ep)$ and is a discontinuous function of $v^\ep$ and not an inclusion. In other words, the value of $w^\ep$ is well specified in the case $v^\ep = 0$ too. This choice is justified also from mathematical point of view, as regularization arguments employed in \cite{porescale} for obtaining the existence of a solution lead to the above form for $w^\ep$.

\subsection{The variational formulation of the microscopic problem}
When defining a weak solution we use common notations in the functional analysis: with $Q$ being either $\Omega$, $\Omega^\ep$, $\Gamma_D$, $\Gamma_G$ or $\Gamma_G^\ep$, we denote by $L^p(Q)$ $(p \geq 1)$ the $p$--integrable functions on $Q$ (in the sense of Lebesgue). The space $H^1_{0,\Gamma_D}(Q)$ restricts the space $H^1(Q)$ of functions having all first order partial derivatives in $L^2$ to those elements vanishing on $\Gamma_D$ (in the sense of traces). Similarly, $W^{k, p}(\Omega)$ contains the functions having the partial derivatives up to the $k^{th}$ order in $L^p$. $(\cdot, \cdot)_Q$ stands for the scalar product in $L^2(Q)$; if $Q = \Omega^\ep$ or $Q = \Omega$, it also denotes the duality pairing between $H^1_{0,\Gamma_D}(Q)$ and $H^{-1}(Q)$ -- the dual of $H^1_{0,\Gamma_D}(Q)$. The corresponding norm is denoted by $\| \cdot \|_{Q}$, or simply $\| \cdot \|$ (where self understood). By $L^\infty(Q)$ we mean functions that are essentially bounded on $Q$, and the essential supremum is denoted by $\|u\|_{\infty, Q}$. Further, for a Banach space $V$ we denote by $L^2(0,T; V)$ the corresponding Bochner space equipped with the standard inner product (where applicable) and norm. Besides, by $\chi_I$ we mean the characteristic function of the set $I$.

Before stating the definition of a weak solution, we introduce the function spaces
\begin{align*}
 \mathcal{U}^\ep &:= \{ u \in L^2 (0,T;H_{0,\Gamma_D}^1 (\Omega^\ep)) \quad : \quad \partial_t u \in L^2 (0,T ; H^{-1}(\Omega^\ep))\},\\
 \mathcal{V}^\ep &:=   H^1 (0,T; L^2(\Gamma_G^\ep)) ,\\
 \mathcal{W}^\ep &:= \{ w \in L^\infty (\get) \quad : \quad 0 \leq w \leq 1\}.
\end{align*}
Then a weak solution is introduced in
\begin{definition}\label{def:weakeps}
A triple $(u^\varepsilon,v^\varepsilon,w^\varepsilon) \in \mathcal{U}^\ep \times \mathcal{V}^\ep \times \mathcal{W}^\ep$ is called a weak solution of \eqref{eq:PH:om}-\eqref{eq:PH:ic} if $u^\varepsilon(0,\cdot ) = u_I, v^\varepsilon(0,\cdot) = v_I,$
\begin{align}
\label{eq:PH:w}
\begin{array}{rcl}
(\partial_t u^\varepsilon,\phi)_{\Omega^{\ep T}} + D (\nabla u^\varepsilon , \nabla \phi)_{\Omega^{\ep T}} - (\textbf{q}^\ep u^\varepsilon, \nabla \phi )_{\Omega^{\ep T}} &=& -\varepsilon (\partial_t v^\varepsilon, \phi)_{\get}, \\
 (\partial_t v^\varepsilon, \theta)_{\get} &=& (r(u^\varepsilon)-w^\varepsilon, \theta)_{\get},
 \end{array}
 \end{align}
for all $(\phi,\theta) \in L^2 (0,T;H^1_{0,\Gamma_D}(\Omega^\ep)) \times L^2 (\get)$, and  $w^\ep$ satisfies \eqref{eq:PH:w=r} a.e. in ${\get}$.
\end{definition}

For the functions appearing as boundary and initial conditions we assume the following
\begin{align*} \label{eq:PH:assump3} \tag{A.3}
u_I \in W_{0,\Gamma_D}^{2, \alpha}(\Omega),
v_I \in H^1(\Omega) , \text{ and } 0 \leq u_I, v_I \leq M_0 \text{ a.e.}, \text{ satisfying }\\
-{\bf \nu} D\nabla u_I = \varepsilon(r(u_I) - w_I),  \text{ (compatibility condition)}.
\end{align*}
The constant $M_0 > 0$ is $\ep$-independent, while $\alpha > 1$. Further, $w_I$ appearing in the compatibility condition satisfies \eqref{eq:PH:w=r}.

Note that the initial precipitation concentration $v_I$ is assumed in $H^1(\Omega)$. For the micro scale model, we consider its trace on $\Gamma_G^\ep$. For simplicity we considered homogeneous conditions on $\Gamma_D$, but the extension to non-homogeneous ones can be carried out without major difficulties. Note that the initial and boundary conditions are compatible, and that the initial conditions are defined for the entire domain $\Omega$.

The existence of weak solutions to \eqref{eq:PH:om}-\eqref{eq:PH:w=r} is proved in \cite{porescale} by regularizing the Heaviside graph. Clearly, the solutions of the regularized problems depend on the regularization parameter. Passing this parameter to zero, one obtains a convergent sequence of solutions; its limit is a weak solution to the original problem, in the sense of Definition \ref{def:weakeps}. Furthermore, the uniqueness of a solution is obtained in \cite{tycho3} by proving the following contraction result with respect to the initial values
\begin{theorem} \label{th:l1}
 Assume \eqref{eq:PH:assump1} and \eqref{eq:PH:assump2} and let $(u^{(i)^\varepsilon}, v^{(i)^\varepsilon},w^{(i)^\varepsilon}) \in \mathcal{U}^\ep,\mathcal{V}^\ep, \mathcal{W}^\ep, i=1,2$ be two solutions in the sense of Definition \ref{def:weakeps}, obtained for the initial values $\; u_I^{(i)}, v_I^{(i)} (i=1,2)$ respectively. Then for any $t \in (0,T]$ it holds
\begin{align}
\nonumber \displaystyle \int\limits_{\Omega^\ep} |u^{(1)\varepsilon}(t,x) - u^{(2)\varepsilon}(t,x)| dx + \varepsilon  \int\limits_{\Gamma_G^\ep} |v^{(1)\varepsilon}(t,s)-v^{(2)\varepsilon}(t,s)| ds \\
 \leq  \displaystyle \int\limits_{\Omega^\ep} |u_I^{(1)}(x) - u_I^{(2)}(x)| dx + \varepsilon  \int\limits_{\Gamma_G^{\varepsilon }} |v_I^{(1)}(s)-v_I^{(2)}(s)| ds
\end{align}
\end{theorem}

\section{The macroscopic model and the main result} \label{sec:macromodel}
In this paper we let $\ep \to 0$ and investigate the limit behaviour of the solutions to the microscopic system \eqref{eq:PH:Stokes}-\eqref{eq:PH:StokesBC}, \eqref{eq:PH:om}-\eqref{eq:PH:ic}. We prove the convergence to the unique solution of the homogenized (macroscopic) system of differential equations defined below. For the flow component, the macroscopic variables $(\textbf{q}, P)$ satisfy the Darcy law \index{Darcy law}
\begin{align}\label{eq:PH:darcy}
\nabla \cdot \textbf{q} = 0,\quad \textbf{q} = -K \nabla P,
\end{align}
for all $x \in \Omega$. The permeability tensor $K$ has the components
\begin{equation}
k_{i j} = \frac{1}{|Y|}\displaystyle \int\limits_{Y} \chi_i^j(y)dy, \qquad \text{ for all } i, j \in \{1, 2, 3\},
\end{equation}
where $\chi_i^j$ is the $i$-th component of $\boldsymbol{\chi}^j =(\chi_1^j, \chi_2^j, \chi_3^j)$ solving the cell problems $(j \in \{1, 2,, 3\})$ \index{Cell problem}
\begin{equation}
(P^D_j)\left\{
\begin{array}{rcll}
- \Delta_y \boldsymbol{\chi}^j(y)&=&\nabla_y \Pi^j(y)+\textbf{e}_j,  & \text{ in } Y\\[0.5em]
\nabla_y \cdot \boldsymbol{\chi}^j(y)&=&0,  &\text{ in } Y\\[0.5em]
\boldsymbol{\chi}^j(y)&=&0,  &\text{ on } \Gamma_G\\[0.5em]
\boldsymbol{\chi}^j, \Pi^j &\text{are}& Z - \text{periodic}.
\end{array}
\right.
\end{equation}
The homogenized model component referring to the chemistry, the solution triple $(u, v, w)$ representing the upscaled solute concentration, precipitate concentration, and the macroscopic dissolution rate are solution to the system 

\begin{align}\label{eq:PH:chem_homog}
\left\{
\begin{array}{rcl}
\partial_t \left(u +  \frac {|\Gamma_G|}{|Y|} v \right) &=& \nabla \cdot \left(S \nabla u -  \textbf{q} u \right), \\
\partial_t v &=& (r(u)- w), \\
w &\in& H(v),
\end{array}
\right.
\end{align}
for all $x \in \Omega$ and $t \in (0, T]$. In addition, analogous to \eqref{eq:PH:w=r}, macroscopic $w$ satisfies,
\begin{align}\label{eq:w=r}
w = \left\{
\begin{array}{ll}
0 \qquad &\mbox{ if } v < 0, \\
\min\{r(u), 1\} &\mbox{ if } v = 0, \\
1 \qquad &\mbox{ if } v > 0.
\end{array}
\right .
\end{align}
The components of the diffusion tensor $S$ are defined by \index{Upscaled diffusion}
\begin{equation}
(S)_{i,j} = D \left[ \delta_{ij}+\frac 1 {|Y|} \displaystyle \int\limits_Y \partial_{y_j} \xi_i dy \right], \qquad \text{ for all } i, j \in \{1, 2, 3\} .
\end{equation}
The functions $\xi_i$ are solutions of the following cell problems $(i \in \{ 1, 2, 3\})$
\begin{equation}
(P^C_i)\left\{
\begin{array}{rcll}
-\Delta \xi_i &=& 0 \; & \text{ in } Y, \\[0.5em]
\boldsymbol{\nu} \cdot \nabla \xi_i &=& \boldsymbol{\nu} \cdot \boldsymbol{e}_i \; & \text{on} \quad \Gamma_G\\[0.5em]
\xi_i &\text{is}& Z  &\text{periodic}.
\end{array}
\right.
\end{equation}
The system \eqref{eq:PH:chem_homog} is complemented by the boundary and initial conditions
\begin{align}
\label{eq:PH:uic} \left \{ \begin{array}{lll}
u(0,\cdot) &=& u_I \quad \text{in} \quad \Omega,\\
v(0,\cdot) &=& v_I \quad \text{in} \quad \Omega, \\
u &=& 0, \quad \text{on} \quad \Gamma_D^T\\
\end{array}
\right.
\end{align}

As for the micro scale model, we are interested in the chemistry component of the upscaled model, for which a weak solution is defined below.
\begin{definition}\label{def:PH:weak_upscaled}
A triple $(u,v,w)$ with $u \in L^2(0,T;H^{1}(\Omega)); \; \partial_t u \in L^2(0,T;H^{-1}(\Omega))$, $v \in L^\infty(0,T;L^2(\Omega)), w \in L^\infty(0,T;L^2(\Omega))$ is called a  weak solution of \eqref{eq:PH:chem_homog}-\eqref{eq:PH:uic} if $(u(0),v(0)) = (u_I,v_I)$,  and
\begin{align}
\nonumber (\partial_t u , \phi)_{\Omega^T} + D(S \nabla u, \nabla \phi)_{\Omega^T} &= - (\textbf{q} u, \nabla \phi)_{\Omega^T} - ( \partial_t v , \phi)_{\Omega^T},  \\
\label{eq:PH:m} (\partial_t v, \theta)_{L^2({\Omega^T })}  &= \left ( r(u)- w , \theta \right )_{L^2({\Omega^T})},\\
\nonumber w &\in H(v), \text{ satisfying } \eqref{eq:w=r}
\end{align}
$\text{ for all } (\phi,\theta) \in L^2(0,T,H_{0,\Gamma_D}^1(\Omega)) \times L^2(0,T;L^2(\Omega))$.
\end{definition}


The main result is as follows:
\begin{theorem} \label{th:PH:main}
As $\ep \searrow 0$, the sequence of micro-scale weak solutions $(u^\ep,v^\ep,w^\ep)$ of problem \eqref{eq:PH:w=r} - \eqref{eq:PH:w}   converges to the unique weak solution $(u,v,w)$ of the upscaled model \eqref{eq:PH:m}. 
\end{theorem}

The notion of convergence will be made more precise in the following sections. We remark that the effective solution $(u,v,w)$ does not depend on the microscopic variable $y \in \Gamma_G$. This results from the fact that initial conditions are considered homogeneous and the processes at the boundaries of perforations are also homogeneous.  Finally, since the flow component is completely decoupled from the chemistry, it is sufficient to quote existing results for the transition from the micro scale (Stokes) model to the upscaled (Darcy) one. In this sense we refer to \cite{All89,  jaeger-mikelic, MA87, SP}.
\section{Uniform estimates for the microscopic solutions}
In this section, we provide estimates for the solutions of the microscopic problem that are uniform with respect to $\ep.$ These will allow passing to the limit $\ep \to 0$, and obtaining the solution to the homogenized model.
In doing so, we recall the a-priori estimates obtained in \cite{porescale}, without considering particularly the homogenization problem.
According to Remarks 2.12 and 2.14 of \cite{porescale}, in the case of a periodically perforated medium (this being the situation here), these estimates are $\ep$-uniform. From \cite{porescale, tycho3} one has:
\begin{theorem}\label{th:estim1}
Assume \eqref{eq:PH:assump1} and \eqref{eq:PH:assump2}, there exists a unique weak solution of  \eqref{eq:PH:om}-\eqref{eq:PH:w=r} in the sense of Definition \ref{def:weakeps}. 
Moreover, this solution satisfies the following estimates
\begin{align}
0 \leq u^\varepsilon, v^\varepsilon \leq M, & \qquad 0 \leq w^\ep \leq 1, \\
\nonumber \|u^\varepsilon\|^2_{L^\infty (0,T;L^2(\Omega^\ep))} &+ \|\nabla u^\varepsilon \|^2_{L^2(\oet)}+\|\partial_t u^\varepsilon\|^2_{L^2(0,T;H^{-1}(\Omega^\ep))}\\
\label{eq:PH:eps_estimates}&+ \varepsilon \|v^\varepsilon\|^2_{L^\infty (0,T;L^2(\Gamma_G^\ep))} + \ep \|\partial_t v^\varepsilon \|^2_{L^2(\get)} \leq C,
\end{align}
where the constants $C>0$ and $M > 0$ are independent of $\varepsilon$.
\end{theorem}

For passing to the limit in the nonlinear reaction terms on $\Gamma_G^\ep$, one needs strong convergence for the solute concentration $u^\ep$. A first step in obtaining this is to extend $u^\varepsilon$ and $\partial_t u^\varepsilon$ from $\Omega^\ep$ to the entire domain $\Omega$. The estimates from Theorem \ref{th:estim1} allow extending $u^\ep$ inside the perforation, but are insufficient for the extension of $\partial_t u^\varepsilon$. In \cite{Jaeger, Maria}, additional estimates for $\partial_t u^\varepsilon$ are obtained by differentiating the model with respect to time. Because of the possible discontinuities in the reaction rate, this approach does not work here.

The approach here is to use the explicit extension procedure in \cite{Jaeger, Maria}, and to estimate the difference quotients with respect to time. With $X$ being a Banach space and for any $h>0$, a function $g: [0,T] \mapsto X$  is extended for negative values of $t$ by
$g(t) = g(0)$.
Recalling (A.3), since the extensions are constant in time, for all $t < 0$ one has
\begin{align}
\label{eq:tneg_eq}
\begin{array}{l}
(\partial_t u^\varepsilon,\phi)_{\Omega^{\ep}} + D (\nabla u^\varepsilon , \nabla \phi)_{\Omega^{\ep}} - (\textbf{q}^\ep u^\varepsilon, \nabla \phi )_{\Omega^{\ep}}+ \varepsilon (r(u^\varepsilon)-w^\varepsilon, \phi)_{\gee} \\[0.5em]
 \qquad \qquad = D (\nabla u_I , \nabla \phi)_{\Omega^{\ep}} - (\textbf{q}^\ep u_I, \nabla \phi )_{\Omega^{\ep}} + \varepsilon (r(u_I)-w_I, \phi)_{\gee}, \\[0.5em]
(\partial_t v^\varepsilon, \theta)_{\gee} - (r(u^\varepsilon)-w^\varepsilon, \theta)_{\gee} = - (r(u_I)-w_I, \theta)_{\gee},
 \end{array}
 \end{align}
for all $\varphi \in H^1_{0,\Gamma_D}(\Omega^\ep)$ and $\theta \in L^2(\gee)$.

{With $g$ being either $u^\varepsilon$ or $v^\varepsilon$ extended to negative times as above, for all $t \geq 0$} we define the difference quotient
\begin{align*}
\triangle_h g(t) := \frac{g(t)-g(t-h)}{h}.
\end{align*}
We have
\begin{lemma} \label{lemma:regularity}
Let $h>0$, $t \in [h, T]$, and $(u^\ep, v^\ep, w^\ep)$ be a weak solution of \eqref{eq:PH:om}-\eqref{eq:PH:w=r} in the sense of Definition \ref{def:weakeps}. 
Then the following estimate is uniform in $h$:
\begin{align*}
\displaystyle \int\limits_{\Omega^\ep} |\triangle_h u^\ep(t,x)| dx +\ep \int\limits_{\Gamma_G^\ep} |\triangle_h v^\ep(t,x)| \leq \int\limits_{\Omega^\ep} |\triangle_h u^\ep(h,x)| dx +\ep  \int\limits_{\Gamma_G^\ep} |\triangle_h v^\ep(h,x)|dx .
\end{align*}
Further, assuming (A.3), a $C > 0$ not depending on $h$ or $\ep$ exists s.t. for any $t \geq 0$ 
\begin{align}
\label{eq:LipContraction} \displaystyle \int\limits_{\Omega^\ep} |\triangle_h u^\ep(t,x)| dx +\ep  \int\limits_{\Gamma_G^\ep} |\triangle_h v^\ep(t,x)| \leq C.
\end{align}
\end{lemma}
\begin{proof}
For the ease of presentation we start with the case $t >h$, where no extension in time is needed. After proving the first part of the lemma we consider the case $t \in [0, h]$.

We follow the $L^1$ contraction proof of Theorem \ref{th:l1} in \cite{tycho3} and define $\mathcal{T}_\delta, \mathcal{S}_\delta: \mathbb{R} \rightarrow \mathbb{R}$
\begin{align*}
\mathcal{T}_\delta(x) := \left \{
\begin{array}{lll}
-x-\frac \delta 2, \quad \text {if} \quad x < -\delta, \\
\frac {x^2} {2\delta}, \quad \text{if} \quad x \in [-\delta, \delta],\\
x-\frac \delta 2 \quad \text{if} \quad x >\delta,
\end{array}
\right.
\qquad \text{ and } \qquad
\mathcal{S}_\delta(x) = \left \{
\begin{array}{lll}
-1, \quad \text{if} \quad x<-\delta,\\
\frac{x}{\delta},\quad \text{if} \quad  x \quad \in [-\delta,\delta],\\
1, \quad \text{if} \quad x > \delta.
\end{array}
\right.
\end{align*}
Here $\delta >0$ is a parameter than can be taken arbitrarily small.
Clearly, $\mathcal{S}_\delta = \mathcal{T}'_\delta$. Note that $\mathcal{T}_\delta$ is a regularized approximation of the absolute value, whereas $\mathcal{S}_\delta$ is the regularized sign function.

Taking $h > 0$ and $t \in (h, T]$ arbitrary, with $(\phi,\theta) \in H^1_{0,\Gamma_D}(\Omega^\varepsilon) \times L^2 (\Gamma_G^{\varepsilon})$ and $\chi_I$ being the characteristic function of the time interval $I$, we test in \eqref{eq:PH:w} first with $\chi_{(h, t)} (\phi, \theta)$, and then with $\chi_{(0, t-h)} (\phi, \theta)$ (both lying in $L^2 (0,T;H_{0,\Gamma_D}^1 (\Omega^\ep)) \times L^2 (\get)$). Subtracting the resulting
gives
\begin{align}\label{eq:PH:wd}
\begin{array}{rcl}
 \displaystyle \int\limits_h^t (\partial_\tau \Delta_h u^\varepsilon,\phi)_{\Omega^\ep} d \tau + D \int\limits_h^t (\nabla \Delta_h u^\varepsilon , \nabla \phi)_{\Omega^{\varepsilon }} d \tau \qquad \qquad & &\\
 \; \; - \displaystyle \int\limits_h^t (\textbf{q} \Delta_h u^\varepsilon, \nabla \phi )_{\Omega^{\varepsilon }} d \tau  + \varepsilon \int\limits_h^t (\partial_t \Delta_h v^\varepsilon, \phi)_{\Gamma_G^{\varepsilon }}, d \tau &=& 0, \\
\displaystyle \int\limits_h^t (\partial_\tau \Delta_h v^\varepsilon, \theta)_{\Gamma_G^\ep} d \tau  - \int\limits_h^t (\Delta_h r(u^\varepsilon)-\Delta_h w^\varepsilon, \theta)_{\Gamma_G^{\varepsilon }} d \tau  &=& 0,  
 \end{array}
 \end{align}
 with $w^\varepsilon \in H(v^\varepsilon) \text{a.e. in} \quad \Gamma_G^\ep.$

{A straightforward step allows replacing the last term on the left in \eqref{eq:PH:wd} by $(\Delta_h r(u^\varepsilon)-\Delta_h w^\varepsilon, \phi)_{\Gamma_G^{\varepsilon }}$.}
With $\phi := \mathcal{S}_\delta(\triangle_h u^\varepsilon)$ and $\theta := \varepsilon  \mathcal{S}_\delta(\triangle_h v^\varepsilon)$, adding the two equations \eqref{eq:PH:wd} gives
\begin{eqnarray}
\nonumber & \displaystyle \int\limits_h^t (\triangle_h \partial_t u^\varepsilon, \mathcal{S}_\delta(\triangle_h u^\varepsilon))_{\Omega^\varepsilon} + \varepsilon  (\triangle_h \partial_t v^\varepsilon, \mathcal{S}_\delta(\triangle_h v^\varepsilon))_{\Gamma_G^\varepsilon} dt \\
\label{eq:PH:w2} & + D \displaystyle \int\limits_{h}^t (\nabla \triangle_h u^\varepsilon, \nabla \mathcal{S}_\delta(\triangle_h u^\varepsilon) )_{\Omega^\varepsilon} dt -  \int\limits_{h}^t (\textbf{q}  \triangle_h u^\varepsilon, \nabla \mathcal{S}_\delta(\triangle_h u^\varepsilon) )_{\Omega^\varepsilon} dt \\
\nonumber & +\varepsilon    \displaystyle \int\limits_h^t  (\triangle_h r(u^\varepsilon) - \triangle_h w^\varepsilon,  \mathcal{S}_\delta(\triangle_h u^\varepsilon)-\mathcal{S}_\delta(\triangle_h v^\varepsilon))_{\Gamma_G^\ep} dt =0.
\end{eqnarray}
Denoting the terms above by $\mathcal{I}_\delta^i,i=1,\dots, 5$, we estimate them separately. $\mathcal{I}_\delta^1$ gives
\begin{align*}
\mathcal{I}_\delta^1 = \displaystyle \int\limits_h^t \int\limits_{\Omega^\varepsilon} \partial_\tau \mathcal{T}_\delta (\triangle_h u^\varepsilon(\tau,x))dx d\tau = \int\limits_{\Omega^\varepsilon} \mathcal{T}_\delta(\triangle_h u^\varepsilon(t,x)) dx - \displaystyle \int\limits_{\Omega^\varepsilon} \mathcal{T}_\delta((\triangle_h u^\ep(h, x)))dx.
\end{align*}
Recall that $0 \leq \mathcal{T}_\delta(s)| \leq |s|+\delta/2$ and $u^\varepsilon(t)\in L^2(\Omega^\varepsilon)$, using the dominated convergence theorem,
\begin{align*}
\lim_{\delta \searrow 0} \mathcal{I}_\delta^1 = \displaystyle \int\limits_{\Omega^\varepsilon} |\triangle_h u(t,x)| dx  -
\int\limits_{\Omega^\varepsilon} |\triangle_h u(h,x)| dx .
\end{align*}
In a similar manner,
\begin{align*}
\lim_{\delta \searrow 0} \mathcal{I}_\delta^2 = \displaystyle \ep \int\limits_{\Gamma_G^\varepsilon} |\triangle_h v(t,x)|
-\ep \int\limits_{\Gamma_G^\varepsilon} |\triangle_h v(h,x)|   .
\end{align*}
Next, since $\mathcal{S}{_\delta}' \geq 0$ a.e. on $\mathbb{R}$, one gets
\begin{align*}
 \mathcal{I}_\delta^3 = \frac{D}{2} \displaystyle \int\limits_h^t \int\limits_{\Omega^\varepsilon} S{_\delta}'(\triangle_h u^\ep) |\nabla \triangle_h u^\varepsilon|^2 dxdt \geq 0.
\end{align*}
Furthermore, for $\mathcal{I}_\delta^4$, since $\textbf{q}$ has zero divergence, using the no-slip boundary conditions together with the vanishing trace of $u^\ep$ on $\Gamma_{D}$ one obtains
\begin{align*}
\mathcal{I}_\delta^4 = \displaystyle \int\limits_{h}^t \int\limits_{\Omega^\varepsilon} \nabla \cdot (\textbf{q}  \mathcal{T}_\delta(\triangle_h u)) = \int\limits_{0}^t \int\limits_{\partial \Omega^\varepsilon} \boldsymbol{\nu} \cdot (\textbf{q} \mathcal{T}_\delta(\triangle_h u)) = 0.
\end{align*}
With $f(u^\ep(t,x),v^\ep(t,x)) = r(u^\ep(t,x))-w^\ep(t,x) \in r(u^\ep(t,x))-H(v^\ep(t,x))$, $\mathcal{I}_\delta^5$ becomes
\begin{align*}
\mathcal{I}_\delta^5 = \ep \displaystyle \int\limits_h^t \int\limits_{\Gamma_G^\varepsilon} \left ( f(u^\ep(t,x),v(t,x)) - f(u^\ep(t-h,x), v(t-h,x)) \right ) (\mathcal{S}_\delta(\triangle_h u^\ep)- \mathcal{S}_\delta(\triangle_h v^\ep)) dx dt.
\end{align*}
Due to the a priori estimates on $u^\ep$ and $v^\ep$ and since $\mathcal{S}_\delta$ is bounded, the integration argument in $\mathcal{I}_\delta^5$ is uniformly dominated in $L^1(\get)$. Therefore, for obtaining uniform estimates for $|\triangle_h u(t,x)|$, it is sufficient to prove that
$$
\lim_{\delta \searrow 0}  \left ( f(u^\ep(t,x),v^\ep(t,x)) - f(u^\ep(t-h,x), v^\ep(t-h,x)) \right ) (\mathcal{S}_\delta(\triangle_h u^\ep)- \mathcal{S}_\delta(\triangle_h v^\ep)) \geq 0
$$
a.e. on $\Gamma_G^{\varepsilon, T}$. This depends on the sign of the difference quotients $\triangle_h u^\ep$ and $\triangle_h v^\ep$. Without loss of generality we only consider the case when $\triangle_h u^\ep \geq 0$, the proof for $\triangle_h u^\ep < 0$ being similar.

Given a pair $(t, x) \in \Gamma_G^{\varepsilon, T}$, we note that if $\triangle_h u^\ep > 0$ and $\triangle_h v^\ep > 0$ one has
$$
\lim_{\delta \searrow 0} (\mathcal{S}_\delta(\triangle_h u^\ep)- \mathcal{S}_\delta(\triangle_h v^\ep)) \rightarrow 0.
$$
The situation is similar if $\triangle_h u^\ep \geq 0$ and $\triangle_h v^\ep \leq 0$. Then we use the monotonicity of $f$ with respect to $u^\ep$ and $v^\ep$ (see also Lemma 1 in \cite{tycho3}) to obtain
$$
f(u^\ep(t,x),v^\ep(t,x)) - f(u^\ep(t-h,x), v^\ep(t-h,x))  \geq 0.
$$
Since $\mathcal{S}_\delta(\triangle_h u^\ep) \geq 0 \geq \mathcal{S}_\delta(\triangle_h v^\ep) $, we have
$$
\lim_{\delta \searrow 0}  \left ( f(u^\ep(t,x),v^\ep(t,x)) - f(u^\ep(t-h,x), v^\ep(t-h,x)) \right ) (\mathcal{S}_\delta(\triangle_h u^\ep)- \mathcal{S}_\delta(\triangle_h v^\ep)) \geq 0.
$$

Using the estimates above into \eqref{eq:PH:w2} gives
\begin{align*}
\displaystyle \int\limits_{\Omega^\ep} |\triangle_h u^\ep(t, x)| dx +\ep \int\limits_{\Gamma_G^\ep} |\triangle_h v^\ep(t, x)| \leq \int\limits_{\Omega^\ep} |\triangle_h u^\ep(h,x)| dx +\ep \int\limits_{\Gamma_G^\ep} |\triangle_h v^\ep(h,x)|dx
\end{align*}
uniformly in $h$.

Finally, we consider the case $t \in [0, h]$. Since the extensions of $u^\ep$, $v^\ep$ and $w^\ep$ satisfy \eqref{eq:tneg_eq}, the steps carried out for $t > h$ lead to
\begin{eqnarray}
\nonumber & \displaystyle \int\limits_0^t (\triangle_h \partial_t u^\varepsilon, \mathcal{S}_\delta(\triangle_h u^\varepsilon))_{\Omega^\varepsilon} + \varepsilon  (\triangle_h \partial_t v^\varepsilon, \mathcal{S}_\delta(\triangle_h v^\varepsilon))_{\Gamma_G^\varepsilon} dt \\
\label{eq:PH:w20} & + D \displaystyle \int\limits_{0}^t (\nabla \triangle_h u^\varepsilon, \nabla \mathcal{S}_\delta(\triangle_h u^\varepsilon) )_{\Omega^\varepsilon} dt -  \int\limits_{0}^t (\textbf{q}  \triangle_h u^\varepsilon, \nabla \mathcal{S}_\delta(\triangle_h u^\varepsilon) )_{\Omega^\varepsilon} dt \\
\nonumber & +\varepsilon    \displaystyle \int\limits_0^t  (\triangle_h r(u^\varepsilon) - \triangle_h w^\varepsilon,  \mathcal{S}_\delta(\triangle_h u^\varepsilon)-\mathcal{S}_\delta(\triangle_h v^\varepsilon))_{\Gamma_G^\ep} dt \\
\nonumber & = \displaystyle - \frac D h \int\limits_0^t (\nabla u_I, \nabla \mathcal{S}_\delta(\triangle_h u^\varepsilon) )_{\Omega^\varepsilon} dt
+ \frac 1 h \int\limits_{0}^t (\textbf{q}  u_I, \nabla \mathcal{S}_\delta(\triangle_h u^\varepsilon) dt \\
\nonumber & \displaystyle
+\frac \varepsilon  h \int\limits_0^t  (r(u_I) - w_I,  \mathcal{S}_\delta(\triangle_h v^\varepsilon) - \mathcal{S}_\delta(\triangle_h u^\varepsilon))_{\Gamma_G^\ep} dt \; =: \mathcal{I}_\delta^6 .
\end{eqnarray}
In view of the boundary conditions for both $u_I$ and $\textbf{q}$ and since $\nabla \cdot \textbf{q} = 0$,  $\mathcal{I}_\delta^6$ rewrites
$$
\mathcal{I}_\delta^6 = \displaystyle \frac D h \int\limits_0^t (\Delta u_I, \mathcal{S}_\delta(\triangle_h u^\varepsilon) )_{\Omega^\varepsilon} dt
- \frac 1 h \int\limits_{0}^t (\textbf{q}  \nabla u_I, \mathcal{S}_\delta(\triangle_h u^\varepsilon) dt 
+\frac \varepsilon  h \int\limits_0^t  (r(u_I) - w_I,  \mathcal{S}_\delta(\triangle_h v^\varepsilon))_{\Gamma_G^\ep} dt .
$$
Using the fact that $u_I \in W^{2, \alpha}$ and the boundedness of the initial condition, of the function $\mathcal{S}\delta$ and of $\textbf{q}$, since $t \leq h$ it follows immediately that a $C > 0$ exists, depending on the initial data but not on $h$ or $\delta$, such that $| \mathcal{I}_\delta^6 | \leq C$ (thus uniformly w.r.t. $h$ and $\delta$). Now proceeding exactly as in the case $t > h$ one obtains
\begin{align}
\label{eq:PH:imp_es} \displaystyle \int\limits_{\Omega^\ep} |\triangle_h u^\ep| dx +\ep \int\limits_{\Gamma_G^\ep} |\triangle_h v^\ep| \leq C .
\end{align}
\end{proof}
\begin{remark}
The fact that $\triangle_h u^\ep$ is bounded uniformly with respect to $h$ in $L^1(\Omega^{\varepsilon})$ norm does not imply $\partial_t u^\varepsilon \in L^1((-h,T) \times {\Omega^\varepsilon})$, since $L^1$ is not reflexive. However, as we will show below, this uniform estimate is sufficient to construct an extension of $u^\varepsilon$ to $\Omega$ having sufficient regularity in time.
\end{remark}
\subsection{Extension results}
In this section, we construct an extension of $u^\ep$ from the $\ep$-dependent domain $(0,T)\times \Omega^\ep$ to the fixed domain $(0,T)\times \Omega$. The difficulty which we have to face here is the fact that the time-derivative of the extended function can not be controlled easily due to the low regularity with respect to time of the original function. Thus, to get the required regularity for the extension of the time derivative, we have to use the special properties of our microscopic solutions, see \eqref{eq:PH:imp_es}.

For $u \in L^1((0,T)\times Y)$, we define the mean value $m_u: (0,T) \rightarrow \mathbb{R}$ as follows
\begin{equation}
m_u(t) := \frac{1}{|Y|} \int_Y u(t,y)dy.
\end{equation}
\begin{lemma} \label{lemma:intermediate}
Let $u \in L^2(0,T;H^{1}(Y)) \cap L^\infty(0,T; L^2(Y))$ satisfying
\begin{equation}\label{difquotu}
\int_Y \left |\frac{u(t,y)-u(t-h,y)}{h}\right | dy \leq C,
\end{equation}
for all $0<h<\frac{T}{4}$, and $t \in (h, T)$. Then $m_u \in W^{1, \infty}(0,T)$, with $||\partial_t m_u||_{L^\infty(0,T)}\leq C$.
\end{lemma}
\begin{proof} 
Since $u \in L^\infty (0,T; L^2(Y))$, it follows immediately that $m_u \in L^\infty (0,T)$. Furthermore, due to \eqref{difquotu}, we have for a.e. $t\in (h,T)$
$$ \left | \frac{m_u(t)-m_u(t-h)}{h}\right |  = \frac{1}{ |Y|} \left | \int_Y \frac{u(t,y) - u(t-h,y)}{h} dy \right | \leq \frac{1}{|Y|} \int_Y \left |\frac{u(t,y)-u(t-h,y)}{h}\right | dy \leq C
$$
with $C$ independent of $h$. Using the properties of difference quotients in Sobolev spaces, see e.g. \cite{GT_Book}, Lemma 7.24, we conclude that $\partial_t m_u \in L^\infty(0,T)$, and $||\partial_t m_u||_{L^\infty(0,T)}\leq C$.
\end{proof}

\begin{lemma} \label{lemma:ext1}
Let $u \in L^2(0,T;H^{1}(Y))$, then there exists an extension $\tilde{u} \in L^2(0,T;H^{1}(Z))$ of $u$, such that
\begin{eqnarray}
&& ||\tilde{u}||_{L^2(0,T;H^{1}(Z))} \leq C ||u||_{L^2(0,T;H^{1}(Y))} \label{extest1}
\end{eqnarray}
\end{lemma}
\begin{proof}
For $t \in (0,T)$, we consider the $H^1$-extension $\tilde{u}(t, \cdot)$ of $u(t, \cdot)$ constructed in \cite{Jaeger}. More precisely, first we extend $u$ into a neighborhood $U$ of $\partial Y_0$ as follows: Using the regularity of $\partial Y_0$, we construct the tubular neighborhood
\begin{eqnarray}
\Phi : \partial Y_0 \times (-\rho,\rho) & \rightarrow & U\\
(\bar{y},\lambda) & \rightarrow & y.
\end{eqnarray}
Using this coordinate transform, we construct an extension of $u$ by reflection:
\[ u^*(y,t) = u^*(\Phi(\bar{y},\lambda),t) =  \left \{ \begin{array}{ccc} u (\Phi(\bar{y}, \lambda),t) & \lambda \geq 0 \\
u (\Phi(\bar{y},-\lambda),t) & \lambda < 0 \end{array} \right. \]
and extend $u^*$ further into $Z$ in any smooth manner. Then, let $\psi : Z  \rightarrow  [0,1]$ be a smooth function with compact support in $Y_0$ and $\psi \equiv 1$ in $Y_0 \setminus U$.
With $m_u$ defined in (\ref{lemma:intermediate}), we define
\begin{equation}\label{tildeu}
\tilde{u} := (1-\psi)(u^* -m_u) + m_u.
\end{equation}
Obviously $\tilde{u}$ is an extension of $u$. To show that $\tilde{u} \in L^2(0,T;H^1(Z))$, we need to prove that for a.e. $t$, $\tilde{u} : [0,T] \rightarrow  H^1(Z)$ is measurable. To do so, let $s^k$ be a sequence of simple functions converging to ${u}$ for a.e. $t$ (as $H^1(Y)$ elements). Extending each $s^k$ to $\tilde{s}^k$ by the procedure in \eqref{tildeu}, the a.e. convergence of $\tilde{s}^k$ to $\tilde{u}$ (now as $H^1(Z)$ elements) will still hold. Finally, by using Lemma 5 in \cite{Jaeger}, we conclude that $\tilde{u} \in L^2(0,T;H^1(Z))$ and \eqref{extest1} is satisfied.
This proves the lemma.
\end{proof}

\begin{lemma} \label{lemma:ext2}
Let  $ u^\ep \in L^2(0,T;H^{1}(\Omega^\ep)) \cap L^\infty(0,T;L^2(\Omega^\ep))$, 
then there exists an extension $\tilde{u}^\ep \in L^2(0,T;H^{1}(\Omega))$  of $u^\ep$ such that
\[ ||\tilde{u} ||_{L^2(0,T;H^{1}(\Omega))} \leq ||u^\ep||_{L^2(0,T;H^{1}(\Omega^\ep))}.\]
\end{lemma}

\begin{proof} 
We use \eqref{extest1} together with a standard scaling argument. For details, see \cite{Jaeger}.
\end{proof}

\begin{lemma}\label{lemma:ext3}
Let $\partial_t u^\ep \in L^2(0,T;H^{-1}(\Omega^\ep))$ and $u^\ep$ satisfies \eqref{eq:LipContraction} for all $0<h<\frac{T}{4}$, and $t \in (h, T)$, then there exists an extension $\partial_t \tilde{u}^\ep \in L^2(0,T;H^{-1}(\Omega)) $ of $\partial_t u^\ep$ such that
\[ ||\partial_t \tilde{u}^\ep||_{L^2(0,T;H^{-1}(\Omega))} \leq C ||\partial_t u^\ep||_{L^2(0,T;H^{-1}(\Omega^{\varepsilon }))}.\]
\end{lemma}
\begin{proof}
Using the improved regularity of $ u^\ep$ with respect to time (see \eqref{difquotu}), we analyze the time derivative of the extension $\tilde{u}^\ep$.
With  $\partial_t m_u$ obtained from Lemma \ref{lemma:intermediate} we define the functional $\partial_t \tilde{u}^\ep \in L^2(0, T; H^{-1}(\Omega))$ by
\begin{equation}\label{partialtutildedistr}
\left < \partial_t \tilde{u}^\ep, \varphi \tilde{\psi} \right >_{\Omega} =
- \int_0^T \int_{\Omega} \tilde{u}^\ep \partial_t \varphi \tilde{\psi} \ dx dt,
\end{equation}
for all $\varphi \in C^\infty_0(0,T)$ and $\tilde{\psi} \in H^1(\Omega)$. By the definition of $\tilde{u}^\ep$, this rewrites
\begin{align*}
\left < \partial_t \tilde{u}^\ep, \varphi \tilde{\psi} \right >_\Omega &= - \int_0^T  \sum_{\ep Z^k \subset \Omega} \left (  \int_{\ep Y^k} u^\ep (\partial_t \varphi) \tilde{\psi} \ dx dt
 + \int_{\ep Z^k \setminus Y^k} \left ((1-\psi)u^*+ \psi m_u\right )  (\partial_t \varphi) \tilde{\psi} \ dx  \right ) dt . 
\end{align*}
Since $\partial_t u^\ep \in L^2(0,T;H^{-1}(\Omega^\ep))$, the first group of integrals are estimated by
\begin{align*}
 \left | \int_0^T  \sum_{\ep Z^k \subset \Omega}   \int_{\ep Y^k} u^\ep (\partial_t \varphi) \tilde{\psi} \ dx dt \right | &\leq C \|\partial_t u^\ep\|_{L^2(0,T;H^{-1}(\Omega^\ep))}
 \|\varphi \tilde{\psi}\|_{L^2(0,T;H^{1}(\Omega))} .
\end{align*}
For the remaining we recall Lemma \ref{lemma:intermediate} to obtain
\begin{align*}
& \left |  \int_{\ep Z^k \setminus Y^k} \left ((1-\psi)u^*+ \psi m_u\right )  (\partial_t \varphi) \tilde{\psi} \ dx  dt \right |  \\
& \qquad \leq \left |  \int_{\ep Y_0^k \cap U^k} \left ((1-\psi)u^*+ \psi m_u\right )  (\partial_t \varphi) \tilde{\psi} \ dx  dt \right |
 +  \left |  \int_{\ep Y_0^k \setminus U^k}  m_u   (\partial_t \varphi) \tilde{\psi} \ dx  dt \right |    \\
& \qquad \leq  C \|\partial_t u^\ep\|_{L^2(0,T;H^{-1}(\Omega^\ep))} \|\varphi \tilde{\psi} \|_{L^2(0,T;H^{1}(\Omega))}.
\end{align*}
The two estimates above prove the lemma.
\end{proof}

\section{Compactness of the microscopic solutions}
First we note down the definitions of two-scale convergence and a lemma that would found to be useful later. Following definitions are standard (e.g. \cite{Allaire,Maria}).
\begin{definition}
A sequence $u^\ep \in L^2(\Omega^\ep)$ is said to converge two-scale to a limit $u \in L^2(\Omega \times Z)$ iff
\begin{align*}\lim_{\ep \searrow 0} \displaystyle \int\limits_{\Omega^\ep} u^\ep(x) \phi(x, \frac{x}{\ep}) dx = \int\limits_\Omega \int\limits_Z u(x,y) \phi(x,y) dx dy \end{align*}
for all $\phi \in \mathcal{D}( \Omega; C_{\text{per}}^\infty(Z))$.
\end{definition}
\begin{definition}
A sequence $v^\ep \in L^2(\Gamma^\ep_G)$ is said to converge two-scale to a limit $v \in L^2(\Omega \times \Gamma_G))$ iff
\begin{align*}\lim_{\ep \searrow 0} \ep \displaystyle \int\limits_{\Gamma_G^\ep} v^\ep(x) \phi(x, \frac{x}{\ep}) dx = \int\limits_\Omega \int\limits_{\Gamma_G} v(x,y) \phi(x,y) dx dy \end{align*}
for all $\phi \in \mathcal{D}( \Omega; C_{\text{per}}^\infty(\Gamma_G))$. \\
\end{definition}
We state the Oscillation Lemma \index{Oscillation Lemma} for functions defined on lower dimensional periodic manifolds (see \cite{Maria} Lemma 1.3.2)
\begin{lemma}\label{l:osc}
For any function $f \in C^0(\bar{\Omega}; C_{\rm{per}}^0(\Gamma_G))$ holds
\begin{align*} \lim_{\ep \searrow 0} \ep \displaystyle \int\limits_{\Gamma^\ep_G} f \left ( x,\frac{x}{\ep} \right ) dx = \int\limits_\Omega \int\limits_{\Gamma_G} f(x,y) dx dy.\end{align*}
\end{lemma}

Based on the estimates proved in the preceding section, the following compactness properties of the microscopic solutions can be shown.
\begin{lemma} \label{th:convergence}
There exists limit functions
\[
\begin{array}{lll}
u \in L^2(0,T;H^1(\Omega)), &  \partial_t u \in L^2(0,T;H^{-1}(\Omega)), & u_1 \in L^2(0,T ; L^2(\Omega ; H_{\rm{per}}^1(Z)),\\
 v \in L^2((0,T) \times \Omega \times \Gamma_G), & \partial_t v \in L^2((0,T) \times \Omega \times \Gamma_G), & w \in L^2((0,T) \times \Omega \times \Gamma_G),
\end{array}
\]
such that up to a subsequence
\begin{enumerate}
\item[1.] $ \tilde u^\ep \rightharpoonup u \; \text{weakly in}\; L^2(0,T;H^1(\Omega)), $
\item[2.] $ \partial_t \tilde u^\ep \rightharpoonup \partial_t u \; \text{weakly in} \; L^2(0,T;H^{-1}(\Omega)),$
\item [3.]$  \tilde u^\ep \rightarrow u \; \text{ strongly  in} \; C^{0}(0,T;H^{-s}(\Omega))\cap L^2(0,T;H^{s}(\Omega)) , s \in (0,1)$
 \item [4.] ${u}^\ep$ two-scale converges to $u$.
 \item [5.]$\nabla u^\ep$ two-scale converges to $\nabla_x u + \nabla_y u_{1}$.
 \item [6.]$v^\ep$ two-scale converges to $v$.
 \item [7.]$\partial_t v^\ep$ two-scale converges to $\partial_t v$.
 \item [8.] $w^\ep$ two-scale converges to $w$.
\end{enumerate}
\end{lemma}
\begin{proof}
The first two results are immediate by using the estimates \eqref{eq:PH:eps_estimates} and the extension lemmas \ref{lemma:ext2} and \ref{lemma:ext3}.  Result item $3$ comes from standard interpolation arguments for Sobolev spaces using  $\tilde {u}^\ep \in L^\infty (0,T;H^1(\Omega))$, and $\partial_t \tilde{u}^\ep \in L^2(0,T;H^{-1}(\Omega))$. In particular, this implies  strong convergence of $u^\ep$ in $L^2(0,T;L^2(\Omega))$. For result item $4$, the compactness arguments (\cite{Allaire,Maria}) imply the two-scale convergence to the same $u$. Compactness results (see \cite{Maria}) for sequences defined on the boundary $\Gamma_G^\ep$, that is for $v^\ep, \partial_t v^\ep$ and $w^\ep$,  yield result items $5$ to $8$.
\end{proof}
Using result item $3$ in Lemma \ref{th:convergence}, a small calculation below shows that $r(u^\ep)$ converges two-scale to $r(u)$. Let $s \in ({1}/{2},1)$. Using the Lipschitz continuity of $r$ and the trace inequality from \cite{anna}, Lemma 4.3, we obtain
\begin{align}
\label{eq:PH:r2s}
\begin{array}{lrc}
 \ep \|r(u^\ep) - r(u)\|_{\get}   \leq C \|u^\ep-u\|_{L^2(0,T; H^s (\Omega^\ep))} \leq C \|u^\ep-u\|_{L^2(0,T; H^s (\Omega))} \searrow 0.
 \end{array}
 \end{align}
This yields
\begin{align*}
\begin{array}{lll}
\left | \displaystyle \int\limits_{\Gamma_G^{\ep T}}  \ep r(u^\ep) \phi(x,\frac{x}{\ep}) dx dt - \displaystyle \int\limits_{\Omega^T} \displaystyle \int\limits_{\Gamma_G} r(u) \phi(x,y) dy dx dt \right|  \leq \\
 \displaystyle \int\limits_{\Gamma_G^{\ep T}} \left | \ep (r(u^\ep) - r(u)) \phi(x,\frac{x}{\ep}) \right | dx dt + \left | \displaystyle \int\limits_{\Gamma_G^{\ep T}} \ep r(u) \phi(x,\frac{x}{\ep}) dx dt - \displaystyle \int\limits_{\Omega^T} \int\limits_{\Gamma_G} r(u) \phi(x,y) dy dx dt \right|
 \end{array}
\end{align*}
and using \eqref{eq:PH:r2s} first term on the right vanishes and the second term tends to zero because of the Oscillation Lemma \ref{l:osc}. Thus, we have shown that  $r(u^\ep)$ converges 2-scale to $r(u)$. \\
Even though $w^\ep$ converges two-scale to $w$,  however this does not provide explicit form for the function $w$. This identification will be obtained by considering the convergence of $v^\ep$ to $v$ in more details. We follow the ideas in \cite{BLM96,CDG08} and use the unfolding operator \index{Unfolding operator}to \index{Unfolding operator!Domain} establish the strong two-scale convergence for $v^\ep$.
\begin{definition}
For a given $\ep >0$, we define an unfolding operator $T^\ep$ mapping measurable functions on $(0,T) \times \Gamma_G^\ep$ to measurable functions on $(0,T) \times \Omega \times \Gamma_G$ by
\begin{align*} T^\ep f(t,x,y) = f(t,\ep {[\frac{x}{\ep}]}+\ep y), \quad y \in \Gamma_G, \quad (t,x) \in (0,T) \times \Omega.\end{align*}
\end{definition}

\begin{remark}
Following \cite{BLM96, CDG08}, the two-scale convergence on $\Gamma_G^\ep$ becomes weak convergence of sequence of unfolded functions on $\Gamma_G \times \Omega$. Besides, the strong convergence of sequence of unfolded functions on $\Gamma_G \times \Omega$  is equivalent to strong two-scale convergence of $v^\ep$ as introduced in \cite{maria-willi}.
\end{remark}
The strong convergence of the unfolded sequence $T^\ep v^\ep$ is provided by the lemmas below.
\begin{lemma}
If $T^\ep v^\ep \rightarrow v^*$ weakly in $L^2((0,T)\times \Omega \times \Gamma_G)$ and $v^\ep$ converges two-scale to $v$ then $v^* = v$ a.e. on $(0,T) \times \Omega \times \Gamma_G$.
\end{lemma}
\begin{proof}{
See Lemma 4.6, \cite{anna} ( see also \cite{BLM96}).
}\end{proof}
\begin{lemma} \label{lemma:2s_strong_vep}
$T^\ep v^\ep$ converges strongly in $L^2((0,T) \times \Omega \times \Gamma_G)$.
\end{lemma}
\begin{proof}{
Let us recall \eqref{eq:PH:w=r}, and note that $w^\ep$ is monotonically increasing with respect to $v^\ep$. This also implies that $T^\ep w^\ep$ is monotone with respect to $T^\ep v^\ep$.
With the change in variable $x \mapsto \ep {[\frac{x}{\ep}]} + \ep y $,$y \in \Gamma_G$ the equation \eqref{eq:PH:w}$_2$ reads on the fixed domain $(0,T) \times \Omega \times \Gamma_G$
\begin{align*}
\partial_t T^\ep v^\ep = T^\ep r(u^\ep) - T^\ep w^\ep.
\end{align*}
 We will prove below that the unfolded sequence $T^\ep v^\ep$ is a Cauchy sequence and hence will converge strongly in $L^2$. Our approach is close to that used in \cite{anna} (also see \cite{maria-willi} for similar results by using translation estimates). The strong convergence of $T^\ep r(u^\ep)$ to $ r(u)$ in $L^2((0,T) \times \Omega \times \Gamma_G)$ and the monotonicity of $T^\ep w^\ep$ will be used to obtain this. Let $m,n$ be two natural number with $n >m$. Now  $T^{\ep_n} v^{\ep_n}-T^{\ep_m} v^{\ep_m}$ satisfies
\begin{align}
\label{eq:PH:vt1}
\begin{array}{lll}
\displaystyle \frac{d}{dt} \|T^{\ep_n} v^{\ep_n}-T^{\ep_m} v^{\ep_m}\|_{L^2(\Gamma_G \times \Omega)}^2 = \\
\displaystyle \int\limits_{\Gamma_G \times \Omega} \left \{ T^{\ep_n} v^{\ep_n}-T^{\ep_m} v^{\ep_m} \right \}  \left \{ T^{\ep_n} r(u^{\ep_n})-T^{\ep_n} w^{\ep_n} - T^{\ep_m} r(u^{\ep_m})+T^{\ep_m} w^{\ep_m} \right \} dx dy.
\end{array}
\end{align}
By monotonicity of $T^\ep w^\ep $ with respect to $T^\ep v^\ep$, we have
\begin{align}
\label{eq:PH:vt2}
\large  ( T^{\ep_n} v^{\ep_n}-T^{\ep_m} v^{\ep_m} \large ) \left (T^{\ep_n} w^{\ep_n}-T^{\ep_m} w^{\ep_m} \right ) \geq 0.
\end{align}
Using \eqref{eq:PH:vt2} in \eqref{eq:PH:vt1}, the right hand side is estimated as
\begin{align*}
\begin{array}{lll}
\displaystyle \frac{d}{dt} \|T^{\ep_n} v^{\ep_n}-T^{\ep_m} v^{\ep_m}\|_{L^2(\Gamma_G \times \Omega )}^2 \\
\leq  \displaystyle \int\limits_{\Gamma_G \times \Omega} \left \{ T^{\ep_n} v^{\ep_n}-T^{\ep_m} v^{\ep_m} \right \} \left \{ T^{\ep_n} r(u^{\ep_n})- T^{\ep_m} r(u^{\ep_m}) \right \} dx dy \\
\leq \displaystyle \frac{1}{2} \|T^{\ep_n} v^{\ep_n}-T^{\ep_m} v^{\ep_m}\|_{L^2(\Gamma_G \times \Omega)}^2 + \frac{1}{2} \|T^{\ep_n} r(u^{\ep_n})- T^{\ep_m} r(u^{\ep_m})\|_{L^2(\Gamma_G \times \Omega)}^2.
\end{array}
\end{align*}
Now integrate in time and notice that as $(n,m) \rightarrow \infty$, due to strong convergence of $T^\ep r(u^\ep)$  the second term goes to $0$ uniformly. Using Gronwall's lemma we conclude that
\begin{align*} \|T^{\ep_n} v^{\ep_n}-T^{\ep_m} v^{\ep_m}\|_{L^2(\Gamma_G^T \times \Omega_h)}^2 \rightarrow 0 \; \mathrm{as} \; n,m \rightarrow \infty\end{align*}
uniformly and hence establishing the strong convergence of $T^\ep v^\ep$ in $L^2((0,T) \times \Omega \times \Gamma_G)$.
}\end{proof}
\begin{remark}{
Note that $w^\ep$ may have discontinuities with respect to $t,x$ which makes dealing with $T^\ep w^\ep$ a delicate task. In the present situation, we are rescued by the fact that $T^\ep w^\ep$ is monotone with respect to $T^\ep v^\ep$ and
 hence, \eqref{eq:PH:vt2} has a good sign which we use in \eqref{eq:PH:vt1}. An alternative approach would be to formulate the boundary conditions as a variational inequality and then use the monotonicity arguments e.g. in \cite{jaeger-mikelic}.
}\end{remark}

\section{Passing to the limit in the microscopic equations.}


Up to now we have obtained the existence of a limit triple $(u,v,w)$ for the sequence $(u^\ep,v^\ep,w^\ep)$. Here we proceed by identifying this limit as the solution of the upscaled system of equations \eqref{eq:PH:chem_homog}, with the initial and boundary conditions \eqref{eq:PH:uic}.
In view of two-scale convergence results, the derivation of limit problem for \eqref{eq:PH:w}$_1$ is standard. We defer the derivation of the limit problem and the cell problem to the end of this section. We begin by considering \eqref{eq:PH:w}$_2$ which contains the nonlinearities. Passing to the limit  as $\ep \searrow 0$ for the left hand side is straightforward. The right hand side is taken care of by using the two scale convergence of $r(u^\ep)$ and the strong convergence obtained in Lemma \ref{lemma:2s_strong_vep}. What remains is to consider \eqref{eq:PH:w}$_3$ and prove that $w \in H(v)$ and has the structure of \eqref{eq:w=r}. Recall that from above discussions, we have the following information:
\begin{align*}
\begin{array}{lll}
T^\ep v^\ep \rightarrow v \quad \mathrm{strongly \; in } \quad L^2((0,T) \times \Omega  \times \Gamma_G) ,\\
T^\ep w^\ep \rightarrow w \quad \mathrm{weakly  \; in} \quad L^2((0,T) \times \Omega  \times \Gamma_G), \\
T^\ep w^\ep \in H(T^\ep v^\ep).
\end{array}
\end{align*}
Since $T^\ep v^\ep \rightarrow v$ strongly in $L^2((0,T) \times \Omega  \times \Gamma_G)$ we have $T^\ep v^\ep \rightarrow v$ a.e.. We have only two situations, either $v(t,x,y) >0$ or $v(t,x,y) =0$. In the first case and with $\mu := v(t,x,y)/2 >0$, the pointwise convergence implies the existence of a $\ep_\mu >0$ such that $T^\ep v^\ep > \mu$ for all $\ep \leq \ep_\mu$. Then for any $\ep \leq \ep_\mu$ we have $T^\ep w^\ep = 1$ implying $w =1$. \\
For the case when $v=0$, we consider the following situations:
\begin{enumerate}
\item [(a)] $u > u^*$ \\
From the pointwise convergence of $T^\ep u^\ep$, there exists an $\ep^*$ such that for $\ep \leq \ep^*$, we have
$u^\ep > u^*$. This gives, using monotonicity of $r$, $r(T^\ep u^\ep) >1$ and recall the definition \eqref{eq:PH:w=r} to obtain $T^\ep w^\ep =1$. This implies that $T^\ep w^\ep \rightarrow 1$ pointwise a.e.
\item [(b)] $u \in [0,u^*)$ \\
Again the pointwise convergence of $T^\ep u^\ep$ implies that for small enough $\ep$, $u^\ep \in (0,u^*)$. In this case, $r(T^\ep u^\ep) <1$ leading to $T^\ep w^\ep = r(T^\ep u^\ep)$ using \eqref{eq:PH:w=r}. With strong convergence of $r$, we get $T^\ep w^\ep$ converges  to $r(u)$ pointwise a.e..
\item [(c)]  $u =u^*$ \\
Using similar arguments as above, $r(u)=1$ and for sufficiently small $\ep$, $r(T^\ep u^\ep) \rightarrow 1$ pointwise a.e.. Hence, $T^\ep w^\ep = \min(r(T^\ep u^\ep),1) \rightarrow 1$ pointwise a.e..
\end{enumerate}
Collecting the above cases, $T^\ep w^\ep$ converges pointwise a.e. to $\widetilde{w}$ where
\begin{align}
\widetilde{w} =
\left \{ \begin{array}{rrr}
1, & v>0,\\
\min(r(u),1), & v =0,\\
0, & v <0.
\end{array}
\right.
\end{align}
Combine this with the weak$-*$ convergence to get $w = \widetilde{w}$ implying that  $w$ has the structure of \eqref{eq:w=r}. This completes the identification of $w$. The above discussions are summarized in the following:
\begin{lemma} \label{lemma:PH:vw2scale}
The two-scale limit functions $v, w$ satisfy
\begin{align*}
\begin{array}{rcl}
\displaystyle (\partial_t v, \theta)_{\Omega^T \times \Gamma_G} &=&  \displaystyle \int\limits_{\Omega^T \times \Gamma_G} \left ( r(u) - w \right ) \theta \qquad \text{ for all } \qquad \theta \in C^\infty(\Omega^T,C^\infty(\Gamma_G)) ,\\
 w &\in &  H(v) \text{ and satisfies } \eqref{eq:w=r}.
 \end{array}
 \end{align*}
\end{lemma}
\begin{remark}
We make an important remark here. Since $u(x,t)$ is independent of the micro-variable $y$, and in view of the initial condition $v_I \in H^1(\Omega)$, we obtain that $v = v(t,x), \; w=w(t,x)$. This independence of $v$ and $w$ from $y$ implies
that the integration over $\Omega \times \Gamma_G$ reduces to integrating over $\Omega$ with the multiplicative factor $|\Gamma_G|$.
\end{remark}
With the above Lemma providing us the limit equations for \eqref{eq:PH:w}$_{2,3}$, we proceed to complete the proof of Theorem \ref{th:PH:main}.
\begin{proof}{
\textbf{Proof of Theorem \ref{th:PH:main}}\\[0.5em]
\noindent
Now we pass to the limit in \eqref{eq:PH:w}$_1$ to obtain the limiting equation and the cell problem. The low regularity of $\partial_t u^\ep$ requires us to  obtain the limiting equations via smooth test functions and using density
arguments. Accordingly, for all $\phi \in C_0^\infty(0,T;H^1(\Omega))$, using partial integration for the time derivative term, the weak formulation \eqref{eq:PH:w}$_1$ gives,
\begin{align} \label{eq:PH:passing_limit}
  -(u^\varepsilon, \chi^\ep \partial_t  \phi)_{\Omega^T} + D (\nabla u^\varepsilon , \chi^\ep \nabla \phi)_{\Omega^T} - (\textbf{q}^\ep \chi^\ep u^\varepsilon, \nabla \phi )_{\Omega^T} = -\varepsilon (\partial_t v^\varepsilon, \phi)_{\get},
 \end{align}
where $\chi^\ep$ is the characteristic function for $\Omega^\ep$. Choose for the test function $\phi(t,x) = \phi_0(t,x)+\ep \phi_1(t,x,\frac{x}{\ep})$ with $\phi_0 \in C_0^\infty(0,T; C_0^\infty(\Omega)) $ and $\phi_1 \in C_0^\infty(\Omega^T;C^{\infty}(Z))$. This gives,
\begin{align*}
\begin{array}{lll}
-\displaystyle \int\limits_{\Omega^T}  u^\ep \chi(\frac{x}{\ep})\left (\partial_t \phi_0(t,x)+\ep \partial_t \phi_1(t,x,\frac{x}{\ep})\right )\\
+D \displaystyle \int\limits_{\Omega^T} \nabla_x u^\ep(t,x) \cdot \chi(\frac{x}{\varepsilon}) \left ( \nabla_x \phi_0(t,x)+\ep \nabla_x \phi_1(t,x,\frac{x}{\ep}) + \nabla_y \phi_1(t,x,\frac{x}{\ep})\right )  \\
+\ep  \displaystyle \int\limits_0^T \int\limits_{\Gamma_G^\ep} \left ( \partial_t v^\ep,   \phi_0(t,x)+\phi_1(t,x,\frac{x}{\ep})  \right )= 0.
\end{array}
\end{align*}
With $\ep \searrow 0$ and using Lemma \ref{th:convergence} and Lemma \ref{lemma:PH:vw2scale}, we obtain
\begin{align*}
- |Y| \displaystyle \int\limits_{\Omega^T}  u \partial_t \phi_0 + D \int\limits_0^T \int\limits_{\Omega \times Y}\left ( \nabla_x u (t,x) + \nabla_y u_{ 1} (t,x,y) \right ) \left ( \partial_x \phi_0(t,x)+\nabla_y \phi_1 (t,x,y) \right ) +\\
\displaystyle \int\limits_0^T \int\limits_{\Omega} \textbf{q} u \nabla \phi_0 + |\Gamma_G| \int\limits_0^T \int\limits_{\Omega } \left ( r(u)-w\right ) \phi_0 = 0.
\end{align*}
Here, to pass to the limit in term containing $\textbf{q}^\ep$, we have used the strong convergence of $\textbf{q}^\ep$ to $\textbf{q}$ as proved in \cite{jaeger-mikelic}. \\
Next, setting $\phi_0 \equiv 0$ we obtain
\begin{align*}
D \displaystyle \int\limits_0^T \int\limits_{\Omega \times Y} \left ( \nabla_x u (t,x) + \nabla_y u_{1} \right ) \cdot \nabla_y \phi_1 (t,x,y) = 0, \text{ for all } \phi_1 \in C_0^\infty(\Omega^T;C^{\infty}(Z)),
\end{align*}
which is a weak form for the cell problem. Further,
\begin{align*} D \displaystyle \int\limits_0^T \int\limits_{\Omega \times Y} \left ( \nabla_x u(t,x)+\nabla_y u_{1} \right ) \nabla_x \phi_0 = D \int\limits_0^T \int\limits_\Omega S \nabla_x \phi_0 \nabla_x u,   \end{align*}
where \begin{align*}(S)_{i,j} = |Y|\delta_{ij}+\displaystyle \int\limits_Y \partial_{y_j} \xi_i; \qquad -\triangle \xi_i = 0 \; \text{in} \; Y, \quad \nabla \cdot \xi_i = \textbf {e}_i \cdot \boldsymbol{\nu} \; \text{on} \; \partial Y. \end{align*}
Now using Lemma \ref{th:convergence}, $\partial_t u \in L^2(0,T;H^{-1}(\Omega))$ and hence, it is justified to perform another partial integration to obtain $(\partial_t u,\phi)$ in the limiting equations. A usual density argument allows us to retrieve the limiting equations for all test functions $\phi \in L^2(0,T;H_0^1(\Omega))$.
Collecting the above results in combination with Lemma \ref{lemma:PH:vw2scale}, we conclude that $(u,v,w)$ is a weak solution as introduced in Definition \ref{def:PH:weak_upscaled}. This completes the proof of Theorem  \ref{th:PH:main}.
}\end{proof}

\section{Uniqueness of the macroscopic model}
\begin{theorem}
Problem \eqref{eq:PH:chem_homog}-\eqref{eq:PH:uic} has a unique solution.
\end{theorem}

\begin{proof}
Assume that there exist two solution triples $(u_1,v_1,w_1)$ and $(u_2,v_2,w_2)$.
Define:
\begin{align*}
U := u_1-u_2, \quad V := v_1-v_2, \quad W:=w_1-w_2 .
\end{align*}
Clearly, at $t=0$, we have $U(0,x) =V(0,x) = W(0,x) =0$ for all $x \in \Omega$. In terms of the differences defined above, we have the resulting equations as:
\begin{align}
\label{eq:PH:uni1} (\partial_t U,\phi)+(D S \nabla U,\nabla \phi)+(\nabla \cdot(\textbf{q} U), \phi) &= - \frac{|\Gamma_G|}{|Y|}(r(u_1)-r(u_2)-W, \phi),\\
 \label{eq:PH:uni2} (\partial_t V, \theta) &= (r(u_1)-r(u_2)-W,\theta),
\end{align}
$ \text{for all } (\phi,\theta) \in L^2(0,T,H_{0,\Gamma_D}^1(\Omega)) \times L^2(0,T;L^2(\Omega)$.

The uniqueness is proved as follows: first we use \eqref{eq:PH:uni2} to estimate $V$ in terms of $U$. This estimate can be then used in \eqref{eq:PH:uni1} to show that for all $t$, the norm of $U(t)$ is bounded by the initial condition, which is zero here. This establishes the uniqueness for $U$ and thereby for $V$ from the previous estimate. The uniqueness for $W$ follows directly from \eqref{eq:w=r}.

Taking $\theta = \chi_{(0,t)} V$ in \eqref{eq:PH:uni2} gives
\[\frac{1}{2} \|V(t,\cdot)\|^2 = \int_0^t \int_{\Omega} (r(u_1)-r(u_2)) V(s,x) dx ds - \int_0^t \int_{\Omega} W V(s,x) dx ds. \]
Since $H(\cdot)$ is monotone, the last term is positive. Using the Lipschitz continuity of $r$, this gives
\[\frac{1}{2} \|V(t,\cdot)\|^2 \leq \frac{1}{2} \int_0^t L_r^2 \|U(s,\cdot)\|^2 ds + \frac{1}{2} \int_0^t \|V(s,\cdot)\|^2 ds  .  \]
Employing Gronwall's inequality one gets
\begin{equation}
\label{eq:PH:univ} \|V(t,\cdot)\|^2 \leq C \exp(t)\int_0^t  \|U(s,\cdot)\|^2 ds \leq C(T) \int_0^t \|U(s,\cdot)\|^2 ds.
\end{equation}
Next, {letting $t \in (0, T]$ fixed arbitrary and with $\psi \in H^1_{0, \Gamma_D}(\Omega)$, combining \eqref{eq:PH:uni1} - \eqref{eq:PH:uni2} and taking $\phi = \chi_{(0,t)} \psi$ in the resulting, since $U$ and $V$ are both 0 at $t=0$ one gets}
$$
(U(t),\psi)+(D S \int_0^t \nabla U(s) ds,\nabla \psi)  + \frac{|\Gamma_G|}{|Y|}(V(t), \psi) =
- (\textbf{q} \int_0^t \nabla U(s) ds, \psi) .
$$
{Here we have used the fact that $\textbf{q}$ is divergence free and does not depend on time. Now we choose }
$\psi(x) = U(t,x)$  to obtain
\begin{align*}
\begin{array}{l}
 \|U(t,\cdot)\|^2 + D \left (S \int_0^t \nabla U(s,\cdot) ds, \nabla U(t,\cdot)\right)+\frac{|\Gamma_G|}{|Y|}\left (V(t,\cdot), U(t,\cdot) \right) \\[0.3em]
\qquad \qquad  
\leq  - (S^{1/2} \int_0^t \nabla U(s) ds, S^{-1/2} \textbf{q} U(t,x)) \\[0.3em]
\qquad \qquad \leq   \frac \mu 2 \left \| S^{1/2}\int_0^t \nabla U(s,\cdot)\right\|^2  + \frac{ 2 M_q^2}{ \mu \alpha_S} \|U(t,\cdot)\|^2 , 
\end{array}
\end{align*}
where $\mu > 0$ is any positive constant. In the above, we have used that $S$ is symmetric positive definite, and hence, there exists $\alpha_S >0$ such that
$(S\boldsymbol{\xi}, \boldsymbol{\xi}) > \alpha_S (\boldsymbol{\xi}, \boldsymbol{\xi})$
 for any $\boldsymbol{\xi} \in \mathbb{R}^3$.

From \eqref{eq:PH:univ} and choosing $\mu$ properly we have 
\begin{align*}
\begin{array}{rcl}
\|U(t,\cdot)\|^2 + \left (S \int_0^t \nabla U(s,\cdot) ds, \nabla U(t,\cdot)\right) \leq C \left ( \int_0^t\|U(z,\cdot)\|^2 dz +  \left \| S^{1/2} \int_0^t \nabla U(s,\cdot)\right\|^2 \right).
\end{array}
\end{align*}
With
\[ E(t) :=  \int_0^t \|U(s,\cdot)\|^2 ds +\frac{1}{2}  \left \| S^{1/2} \int_0^t \nabla U(s,\cdot) ds \right\|^2,  \]
the above becomes 
\[ E^\prime(t) \leq C \ E(t) .\]
Clearly, $E(0) =0$ and $E(t) \geq 0$ for all $t$, which immediately gives $E (t) = 0$ for all $t$. This ensures that $ U (t) =  0 $  and, by \eqref{eq:PH:univ}, $V (t) = 0$. This concludes the proof of uniqueness.
\end{proof}


\section*{Acknowledgement}
The work of K. Kumar was supported by the Technology Foundation STW through Project 07796.  This support is gratefully acknowledged. The authors are members of the International Research Training Group NUPUS funded by the German Research Foundation DFG (GRK 1398) and by the Netherlands Organisation for Scientific Research NWO (DN 81-754). The authors would like to thank Profs. W. J{\" a}ger (Heidelberg) and A. Mikeli{\' c} (Lyon) for their suggestions and advices.

\bibliographystyle{plain}
\bibliography{All}

\end{document}